\documentclass{amsart}
\usepackage[utf8]{inputenc}
\usepackage{graphicx}
\usepackage[margin=1in]{geometry}
\usepackage[colorinlistoftodos,textsize=footnotesize]{todonotes}
\usepackage{amsmath}
\usepackage{amsthm}
\usepackage{amsfonts}
\usepackage{amssymb}
\usepackage{youngtab}
\usepackage{hyperref}
\usepackage{tikz}
\usepackage{subcaption}
\graphicspath{ {./images/} }

\title{RSK-complete Cycle Decompositions}
\author{Agastya Goel} 
\address{(Agastya Goel): Euler Circle, Mountain View, CA 94040}
\email{goel.agastya@gmail.com}
\author{Simon Rubinstein-Salzedo}
\address{(Simon Rubinstein-Salzedo): Euler Circle, Mountain View, CA 94040}
\email{simon@eulercircle.com}
\date{\today}

\newcommand{\lf}{\left\lfloor}
\newcommand{\rf}{\right\rfloor}
\newcommand{\lc}{\left\lceil}
\newcommand{\rc}{\right\rceil}
\newcommand{\lp}{\left(}
\newcommand{\rp}{\right)}

\newtheorem{theorem}{Theorem}
\newtheorem{remark}[theorem]{Remark}
\newtheorem{lemma}[theorem]{Lemma}

\theoremstyle{definition}
\newtheorem{definition}[theorem]{Definition}

\theoremstyle{remark}
\newtheorem*{example}{Example}

\numberwithin{theorem}{section}

\begin{document}

\maketitle

\begin{abstract}
    We characterize the class of cycle decompositions that can achieve all Young tableau shapes (except the trivial ones with a single row or a single column) under the Robinson--Schensted--Knuth (RSK) correspondence, a property that we call RSK-completeness. We prove that for even $n$, cyclic permutations comprise the only fixed cycle decomposition that is RSK-complete. For odd $n$, cyclic permutations and almost cyclic permutations which have a cycle of length $n-1$ are the only RSK-complete cycle decompositions.
\end{abstract}

\section{Introduction}

The Robinson--Schensted--Knuth (RSK) correspondence, described in~\cite{robinson,schensted_1961,pjm}, is a way of mapping a sequence (typically, one representing a permutation) to a pair of Young tableaux. See~\cite{fulton} for a thorough description of the RSK correspondence and its main application, to the representation theory of the symmetric group. Both tableaux have the same shape, which we will call the \emph{RSK shape} of the sequence. The RSK correspondence has found numerous uses in combinatorics and group theory; see for instance~\cite{alma991004256029702915, kostka, Fomin1988}.

This correspondence encodes certain permutation statistics in easily visible ways. The one most directly relevant to our work is that the length of the longest increasing (respectively, decreasing) subsequence  of the permutation is the same as the length of the first row (respectively, column) of its RSK shape, described in~\cite{schensted_1961}. Further, this can be extended to the length of the largest possible union of $k$ increasing/decreasing subsequences, as shown in~\cite{GREENE1974254}. Given that some of the most interesting properties of the RSK-correspondence relate to the shape of the resulting tableaux, much work has been done in the past relating certain classes of permutations, such as Boolean permutations~\cite{GUNAWAN2022} and pattern-avoiding permutations~\cite{LEWIS20111436, SIMION1985383, Yan2012, Ouchterlony05}, to their RSK shapes. In this paper, we will consider classes of permutations based on their cycle decompositions.

We start with the following natural question: {\em what classes of permutations can produce all possible RSK shapes?} The RSK shapes that consist of a single row or a single column can only be achieved using the identity permutation and the reverse permutation respectively, so we will leave these two trivial cases aside. Then, we will define a class of permutations over $\{1, 2, \dots, n\}$ to be {\em RSK-complete} if permutations in this class can generate all remaining RSK shapes of that size.

Involutions are also easily visible under the RSK correspondence: the involutions are exactly the permutations for which the two RSK tableaux are the same~\cite{schutzenberger}. This means that the class of involutions is RSK-complete. Our focus in this paper will be on permutations with a {\em fixed} cycle decomposition. Specifically, we show that cyclic permutations, which have the simple cycle decomposition $(n)$, are RSK-complete, and for even $n$, this is the the only RSK-complete cycle decomposition. For odd $n$, there is just one more RSK-complete cycle decomposition, namely ``almost cyclic'' permutations with cycle decomposition $(n-1,1)$, i.e.\ those that are $(n-1)$-cycles. Thus, we obtain a complete characterization of RSK-complete cycle decompositions. 


In Section~\ref{sec:prelim}, we define the RSK correspondence and RSK-completeness formally and establish notation that we will use for the rest of the paper. In Section~\ref{sec:cyclic}, we show that cyclic permutations are RSK-complete. Section~\ref{sec:almost} shows that almost cyclic permutations are RSK-complete only for odd $n$. In Section~\ref{sec:converse}, we show that no other fixed cycle decomposition can be RSK-complete, thus yielding a complete characterization of RSK-complete cycle decompositions. 

In Section~\ref{sec:almost}, we also prove an additional connection between  almost cyclic permutations and the class of realizable RSK shapes when $n$ is even: that the class of realizable shapes includes everything except the two trivial shapes and the shape containing just two rows of size $\frac{n}{2}$ each. An interesting problem for future work is to characterize the class of RSK shapes achievable by any given cycle decomposition.


\section{Preliminaries}
\label{sec:prelim}


\begin{definition}
A \emph{Young diagram} is a subset of cells of a grid, with the cells in each row being left-justified (i.e.\ the row is filled from left to right), and the row lengths being nonincreasing. A \emph{Young tableau} is obtained by taking a Young diagram and filling each cell with a positive integer. We then say a Young tableau is \emph{semistandard} if its rows are weakly increasing and its columns are strictly increasing. If there are $n$ cells in our tableau and each cell is filled with a distinct integer in $\{1, 2, \ldots, n\}$ in addition to being semistandard, we say that the tableau is a \emph{standard Young tableau}, or \emph{SYT}.
\end{definition}

In this paper, we will be concerned primarily with studying the shapes of Young tableaux, which is also the underlying Young diagram:

\begin{definition}
The \emph{shape} of a Young tableau is its list of row lengths, sorted in nonincreasing order.
\end{definition}

Thus the shape of a Young tableau with $n$ elements is a partition of $n$. We can write a Young tableau as a sequence of sequences. For example, the SYT in Figure~\ref{fig:syt-basic} corresponds to the sequence $((1, 4, 5), (2, 6), (3))$, and
figure~\ref{fig:syt-shape} represents its shape.
For a Young tableaux $\mathfrak{T}$, let $\mathfrak{T}_r$ represent the $r^\text{th}$ row of $\mathfrak{T}$, let $\mathfrak{T}_{r, c}$ represent the $c^\text{th}$ element in $\mathfrak{T}_r$, and let $\Gamma(\mathfrak{T})$ represent its shape.

\begin{figure}
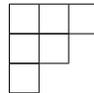

\begin{subfigure}{.45\textwidth}
  \centering
  \[\young(145,26,3)\]
  \caption{The SYT $((1, 4, 5), (2, 6), (3))$.}
  \label{fig:syt-basic}
\end{subfigure}
\begin{subfigure}{.45\textwidth}
  \centering
  \[\yng(3,2,1)\]
  \caption{The corresponding shape (3, 2, 1).}
  \label{fig:syt-shape}
\end{subfigure}
\caption{A simple standard Young tableau and its shape.}
\end{figure}

\begin{definition}
Given two sequences $\gamma^1$ and $\gamma^2$, we will let $\gamma^3 = \gamma^1+\gamma^2$ denote the sequence of length $\max(|\gamma^1|, |\gamma^2|)$ where $\gamma^3_i = \gamma^1_i+\gamma^2_i$ for $1 \leq i \leq \max(|\gamma^1|, |\gamma^2|)$. We implicitly let $\gamma_i = 0$ for $i > |\gamma|$.
\end{definition}

For example, $(3, 2, 1) + (1, 1) = (4, 3, 1)$.

Given any sequence $\sigma$, we will use $\sigma_i$ to refer to the $i^\text{th}$ element, and $\sigma[\ell, r]$ to refer to the subsequence $(\sigma_\ell,\dots,\sigma_r)$. We will represent a permutation over $\{1, 2, \ldots, n\}$ as a sequence, with $\sigma_i$ being the element in the $i^\text{th}$ position in the permutation.

Finally, we define  the \emph{Robinson--Schensted--Knuth correspondence}~\cite{robinson,schensted_1961,pjm} (or simply the RSK correspondence), which is the main topic of the paper. This is defined as two functions $P$ and $Q$, which take in a sequence as input. $P$ returns a tableau containing the elements of the sequence, and $Q$ returns one containing the elements $\{1, 2, \ldots, n\}$, where $n$ is the length of the sequence. We will not use $Q$ in this paper, and we are further only interested in SYTs. Thus, we will only consider the first tableau in the pair, and only the case where the integers in the sequence are distinct.\footnote{We will primarily be interested in applying this correspondence to permutations; however, in the process of constructing these correspondences, we will also have to consider subsequences of permutations, hence the use of sequences of distinct integers.}

\begin{definition}The \emph{Robinson--Schensted--Knuth correspondence} denoted by $P$ maps a sequence $\sigma$ of distinct integers to a Young tableau. It can be defined recursively as follows: First, define $P(\sigma[1,1])$ to be the SYT with only one row which contains exactly one cell, filled with the element $\sigma_1$. Now, assume that $\mathfrak{T} = P(\sigma[1, c])$ has already been computed. To compute $P(\sigma[1, c+1])$, we use the displacement procedure. If the element to insert is $a$, and $a$ is larger than the last element in the first row of $\mathfrak{T}$, put $a$ at the end of the first row and terminate. Else, let $b$ be the leftmost element greater than $a$ in the first row. Replace $b$ with $a$ and insert $b$ into the next row recursively.
\end{definition}

\begin{example} Consider the permutation $\sigma = (2, 5, 1, 4, 6, 3)$. Then, $P(\sigma[1, 4])$ is the following tableau:
\[\young(14,25).\]
When we add $6$, it is the largest element in the first row, so we simply add it to the end of the first row and stop. Hence $P(\sigma[1, 5])$ is
\[\young(146,25).\]
When we next add $3$, it displaces $4$ from the first row, which displaces $5$ from the second row, which in turn starts a new row, and we get the following tableau as $P(\sigma)$:
\[\young(136,24,5).\] \end{example}

We also will use the following theorem about SYTs.

\begin{theorem}[Schensted~\cite{schensted_1961}]
\label{Schensted’s Theorem}
The length of the first row of $P(\sigma)$ is equal to the longest increasing subsequence (LIS) of $\sigma$, and the length of the first column of $P(\sigma)$ is equal to the longest decreasing subsequence (LDS) of $\sigma$.
\end{theorem}

Given the concrete connection between the LIS/LDS and the RSK-shape of a sequence, it is natural to ask which class of permutations can generate a given class of shapes. From Schensted’s Theorem, it immediately follows that the trivial shapes $(n)$ and $(1, 1, \dots, 1)$ can only be generated by the identity permutation and the reverse permutation, respectively. This motivates the definition of {\em RSK-complete} classes of permutations. Here $S_n$ represents the $n^\text{th}$ symmetric group, i.e.\ all permutations of $\{1, 2, \dots, n\}$.

\begin{definition}
A set of permutations $S \subseteq S_n$ is said to be \emph{RSK-complete} if any partition of $n$, other than the trivial partitions $(n)$ and $(1, 1, \dots, 1)$, is $\Gamma(P(\sigma))$ for some $\sigma\in S$.
\end{definition}

\begin{definition}
A cycle decomposition $\sigma = (n_1, n_2, \ldots, n_k)$, where $n_1 \ge n_2 \ge \dots \ge n_k$ and $\sum_{i=1}^k n_i = n$, is said to be \emph{RSK-complete} if the set of all permutations with cycle decomposition $\sigma$ is RSK-complete.
\end{definition}

\begin{definition}

A permutation $\sigma$ is said to be a \emph{cycle} or a \emph{cyclic permutation} if it has a cycle decomposition of $(|\sigma|)$.
\end{definition}

\begin{definition}

A sequence $\sigma$ of size $n$ is said to be a \emph{path} if there exists some permutation $\sigma'$ of size $n$ such that for all $1 \leq i < n$, $\sigma_{\sigma'_i} = \sigma'_{i+1}$, but $\sigma$ itself is not a permutation.
\end{definition}

Informally, a sequence is a path if its corresponding directed graph forms a path, with exactly one of its elements being greater than $n$.

\begin{definition}

The subarray $[\ell, r]$ of a sequence $\sigma$ is a \emph{shifted path} if there is some permutation $\sigma'$ of $\{\ell, \ldots, r\}$ such that for all $\ell \leq i < r$, $\sigma_{\sigma'_i} = \sigma'_{i+1}$.

\end{definition}

Informally, a shifted path is a part of a permutation that is a path when considered in the context of the permutation's graph.

We will now introduce the idea of a displacement path.

\begin{definition}
The \emph{displacement path} of an element $x$ that is being inserted into a SYT $\mathfrak{T}$ is the sequence of positions modified when that element is inserted. Since the row indices are consecutive, we can simply write the sequence of column indices. We will use $\Delta_\mathfrak{T}(x)$ to denote this sequence, and we drop the subscript $\mathfrak{T}$ when it is clear from the context. Further, we will let $\Lambda_\mathfrak{T}(x)$ be the sequence of displaced elements, so $(\Lambda_\mathfrak{T}(x))_i = \mathfrak{T}_{i, \Delta(x)_i}$, for $1 \leq i < |\Delta(x)|$. Again, we will often drop the subscript.
\end{definition}

In the example above, the displacement path of element 6 is $(3)$, the displacement path of element 3 is $(2, 2, 1)$, and $\Lambda(3) = (4, 5)$. Note that $\Lambda(3)$ is increasing. Displacement paths will prove useful in analyzing the properties of SYTs as they are being constructed.

We will next define tail-monotone permutations, where the largest elements of the permutation are in decreasing order towards the end, with one possible exception. These sequences will be useful for incremental construction of SYTs with a desired shape.

\begin{definition}
A permutation $\sigma$ of length $m = n+k$ is said to be $(n,k)$-{\em tail-monotone} if the following are true:
\begin{itemize}
    \item The subsequence $\sigma[n+1, n+k]$ is a descending sequence.
    \item $\sigma_{n+1} = m$.
    \item There is at most one $i \le n$ such that $\sigma_i > n$.
\end{itemize}
\end{definition}

\begin{theorem}
\label{Descending Sequence Concatenation}
Let $\sigma$ be a $(n,k)$-tail-monotone permutation. Then \[\Gamma(P(\sigma)) = \Gamma(P(\sigma[1, n])) + \Gamma(P(\sigma[n+1, n+k])) = \Gamma(P(\sigma[1, n])) + (1^k),\]
where $(1^k)$ represents a sequence of $k$ $1$'s.
\end{theorem}

That is, the shape of $P(\sigma)$ is the shape of $P(\sigma[1,n])$, with an extra cell in each of the first $k$ rows. For example, if $\sigma = (4, 1, 2, 7, 3, 8, 6, 5)$, which is $(5,3)$-tail-monotone, then $P\lp\sigma[1, 5]\rp$ is
\[\young(123,47),\]
and $P\lp\sigma\rp$ is
\[\young(1235,468,7).\]

Observe that the second tableau is not simply obtained by adding an element from the tail to each row. We will prove Theorem~\ref{Descending Sequence Concatenation} using displacement paths.

\begin{lemma}
\label{Displacement Path Left Shift}
Let $\sigma$ be a $(n,k)$-tail-monotone permutation.
For all $z,r$ such that $1 < z \leq k$ and $1 \leq r \leq |\Delta(\sigma_{n+z-1})|$, we must have $\Delta(\sigma_{n+z})_r \leq \Delta(\sigma_{n+z-1})_r$.
\end{lemma}

\begin{proof}
Note that this is equivalent to saying that $\Lambda(\sigma_{n+z})_r \leq \Lambda(\sigma_{n+z-1})_r$. We will proceed via induction on $r$. For our base case, consider $r = 1$. Since $\sigma_{n+z} < \sigma_{n+z-1}$, $\sigma_{n+z}$ will end up displacing an element with a lower index than $\sigma_{n+z-1}$ will in the first row. Thus, $\Lambda(\sigma_{n+z})_1 \leq \Lambda(\sigma_{n+z-1})_1$.

For the inductive step, $\Lambda(\sigma_{n+z})_r \leq \Lambda(\sigma_{n+z-1})_r$ is equivalent to saying the element displaced to row $r+1$ by $\sigma_{n+z}$ is less than or equal to the element displaced by $\sigma_{n+z-1}$. A smaller element will displace a smaller element, so $\Lambda(\sigma_{n+z})_{r+1} \leq \Lambda(\sigma_{n+z-1})_{r+1}$ as well.
\end{proof}

\begin{lemma}
\label{Displacement Path Expansion}
Let $\sigma$ be a $(n,k)$-tail-monotone permutation.
For all $z$ such that $1 < z \leq k$, we must have $|\Delta(\sigma_{n+z})| = |\Delta(\sigma_{n+z-1})|+1$.
\end{lemma}

\begin{proof}

We will prove this in two parts, first proving that $|\Delta(\sigma_{n+z})| \geq |\Delta(\sigma_{n+z-1})|+1$ and then that $|\Delta(\sigma_{n+z})| \leq |\Delta(\sigma_{n+z-1})|+1$.

For the first inequality, let $r = |\Delta(\sigma_{n+z-1})|$. Then, by Lemma \ref{Displacement Path Left Shift}, $\Delta(\sigma_{n+z})_r \leq \Delta(\sigma_{n+z-1})_r$. Therefore, the insertion of $\sigma_{n+z}$ will not terminate at row $r$, and $|\Delta(\sigma_{n+z})| > r = |\Delta(\sigma_{n+z-1})|$.

Let us prove the second inequality by contradiction. Suppose that $|\Delta(\sigma_{n+z})| > |\Delta(\sigma_{n+z-1})|+1$. Then, since there is only one $1 \leq p \leq n$ for which $\sigma_p > \sigma_{n+z}$, it must be that $|\Delta(\sigma_{n+z})| = |\Delta(\sigma_{n+z-1})|+2$. If $\ell = |\Delta(\sigma_{n+z-1})|$, this only happens when the last element of $\mathfrak{T}_{\ell+1}$  is $\sigma_p$. However, note that $(\mathfrak{T}_{1, |\mathfrak{T}_1|}, \mathfrak{T}_{2, |\mathfrak{T}_2|}, \ldots, \mathfrak{T}_{\ell,|\mathfrak{T}_\ell|}) = (\sigma_{n+z-1}, \sigma_{n+z-2}, \ldots, \sigma_{n+1})$. Thus, $\sigma_{n+1} = n+k$ would have to displace $\sigma_p$, which is impossible. Thus, we have a contradiction.
\end{proof}

We can now prove Theorem \ref{Descending Sequence Concatenation}.

\begin{proof}[Proof of Theorem \ref{Descending Sequence Concatenation}]
By induction using Lemma \ref{Displacement Path Expansion}, we have $|\Delta(\sigma_{n+z})| = z$. Thus, $|\Delta(\sigma_{n+k})| = k$, meaning that our final shape is $\Gamma(P(\sigma[1, n]))+(1^k)$, as desired.
\end{proof}

\section{Cyclic Permutations}
\label{sec:cyclic}
In this section, we will prove that the cycle decomposition $(n)$ is RSK-complete:

\begin{theorem}
\label{Cyclic Theorem}
All RSK shapes apart from $(1, 1, \ldots, 1)$ and $(n)$ can be achieved using only cyclic permutations, where a permutation is cyclic if its cycle decomposition consists of exactly one cycle.
\end{theorem}

Our proof will be constructive. To construct a cyclic permutation with a given shape, we will start with one of two base case constructions: the first construction gives us all shapes with two columns, and the second gives all shapes where all boxes are either in the first row or the first column. Then we will apply another construction to successively add columns by concatenating descending sequences until we have the desired shape.

First, let us begin with the following sequence construction, which will be the main building block for the two-column base case as well as our construction for adding columns.

\begin{definition}
For an integer $n \geq 1$, let $L(n)$ be defined as follows:
\[L(n) = \lp n, n-1, \ldots, \lc\frac{n}{2}\rc+1, \lc\frac{n}{2}\rc-1, \ldots, 2, 1\rp.\]
Note that the length of $L(n)$ is $n-1$ since it is missing the element $\lc n/2 \rc$.
\end{definition}

\begin{lemma}
\label{Descending Sequence Path}
Consider the directed graph $G = (V, E)$, where $V = \{ 1, 2, \ldots, n-1\}$, and $E = \{(i, L(n)_i):1<i<n\}$. Then $G$ forms a path.
\end{lemma}

\begin{proof}
Consider $G$ in the form of a bipartite graph, with nodes $\lp1, 2, \ldots, \lf\frac{n}{2}\rf\rp$ on the left, and nodes $\lp\lf\frac{n}{2}\rf+1, \lf\frac{n}{2}\rf+2, \ldots, n-1\rp$ on the right. Let $\ell_i = \lf\frac{n}{2}\rf+1-i$ for $1 \leq i \leq \lf\frac{n}{2}\rf$, and let $r_i = \lf\frac{n}{2}\rf+i$ for $1 \leq i \leq \lc\frac{n}{2}\rc-1$. Furthermore, let $p$ be $1$ for odd $n$ and $0$ for even $n$, and let $N(i) = L(n)_i$.

For all $j$ where $1 \leq j \leq \lf\frac{n}{2}\rf$, we have that $N(j) = n+1-j$. Thus, $N(\ell_i) = n+1-\ell_i = n+i-\lf\frac{n}{2}\rf = r_{i+p}$. Similarly, for some $j$ where $\lf\frac{n}{2}\rf < j < n$, $N(j) = n-j$. Thus, $N(r_i) = n-r_i = n-i-\lf\frac{n}{2}\rf = \ell_{i+1-p}$.
Hence for any $i$ where $1 \leq i < \lf\frac{n}{2}\rf$ we have $N(N(\ell_i)) = \ell_{i+1}$, and similarly for any $i$ where $1 \leq i \leq \lc\frac{n}{2}\rc - 2$ we have $N(N(r_i)) = r_{i+1}$. Thus, $G$ is acyclic. As there is exactly one node of indegree $0$ (namely, node $\lc \frac{n}{2} \rc$), and one node of outdegree $0$ (namely, node $1$), and all other nodes have in and outdegree $1$, our graph forms a path.
\end{proof}

For example, $L(7) = (7, 6, 5, 3, 2, 1)$, and the corresponding directed graph as defined in Lemma \ref{Descending Sequence Path} is shown in Figure \ref{lambda7 example}.

\begin{figure}
    \centering
\begin{tikzpicture}
\filldraw[fill=gray!15!white, draw=gray] (4,4) circle (0.3cm) node {3};
\filldraw[fill=gray!15!white, draw=gray] (8,4) circle (0.3cm) node {4};
\draw[-latex] (7.7, 4) -- (4.3, 4);
\draw[-latex] (4.26, 3.87) -- (7.74,2.13);
\filldraw[fill=gray!15!white, draw=gray] (4,2) circle (0.3cm) node {2};
\filldraw[fill=gray!15!white, draw=gray] (8,2) circle (0.3cm) node {5};
\draw[-latex] (7.7, 2) -- (4.3, 2);
\draw[-latex] (4.26, 1.87) -- (7.74,0.13);
\filldraw[fill=gray!15!white, draw=gray] (4,0) circle (0.3cm) node {1};
\filldraw[fill=gray!15!white, draw=gray] (8,0) circle (0.3cm) node {6};
\draw[-latex] (7.7, 0) -- (4.3, 0);
\end{tikzpicture}
    \caption{The directed graph corresponding to $L(7)$.}
    \label{lambda7 example}
\end{figure}
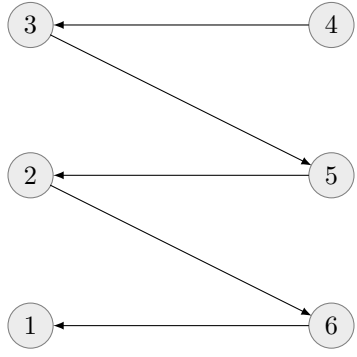

We also need to introduce a key operation on sequences.

\begin{definition}
Given a sequence $\sigma$, let $I(\sigma)$ be the sequence $\sigma'$ such that $|\sigma'| = |\sigma|$ and for all $1 \leq i \leq |\sigma|$, $\sigma'_i = \sigma_i+1$. We will further define $I^k(\sigma)$ to denote  $\sigma$ for $k=0$ and $I\lp I^{k-1}(\sigma)\rp$ for $k > 0$, where $k$ is an integer.
\end{definition}

We will now present the following construction. We will use $\oplus$ to denote the concatenation operator. For example, $(2, 4) \oplus (1, 3) = (2, 4, 1, 3)$.

\begin{definition}
For two integers $1 < m < n$, we will define $B(m, n)$ to be the following concatenation of sequences:

\[B(m, n) = (2, 3, \ldots, m-1) \oplus\lp\lf\frac{n+m}{2}\rf\rp \oplus I^{m-1}\lp L(n-m+1)\rp \oplus (1).\]
\end{definition}

\begin{lemma}
\label{B(m, n) is a permutation}
For any two integers $1 < m < n$, the sequence $B(m, n)$ is a permutation of length $n$.
\end{lemma}

\begin{proof}
We wish to show that every integer from $1$ to $n$ appears exactly once in $B(m, n)$. We can reorder $B(m, n)$ into four sequences for clarity. Namely, $(1)$, $(2, 3, \ldots, m-1)$, $I^{m-1}\lp L(n-m+1)\rp$, and $\lp\lf\frac{n+m}{2}\rf\rp$. Clearly, the first two sequences give us the elements $\{1, 2, \ldots, m-1\}$. By its definition, $L(n-m+1)$ contains the elements $\{1, 2, \ldots, \lc\frac{n-m-1}{2}\rc, \lc\frac{n-m+3}{2}\rc, \ldots, n-m, n-m+1\}$. Hence, $I^{m-1}\lp L(n-m+1)\rp$ contains the elements
\begin{align*}
    \left\{m, m+1, \ldots, \lc{\frac{n+m-3}{2}}\rc, \lc{\frac{n+m+1}{2}}\rc, \ldots, n-1, n\right\}
    &= \{m, m+1, \ldots, n\} \setminus \left\{\lc \frac{n+m-1}{2}\rc\right\} \\
    \\ &= \{m, m+1, \ldots, n\} \setminus \left\{\lf \frac{n+m}{2} \rf\right\}.
\end{align*}
Thus, our final two sets taken together yield each value in $\{m, m+1, \ldots, n\}$, showing that $B(m, n)$ is a permutation.
\end{proof}

\begin{lemma}
\label{Base Case Validity}
The permutation $B(m, n)$ is cyclic, and $\Gamma\lp P\lp B(m, n)\rp\rp = (m, 1, \dots, 1)$, with $n-m$ $1$'s.
\end{lemma}

\begin{proof}
First we prove cyclicity. Clearly, $B(m, n)[1, m-1]$ forms a path, and by Theorem \ref{Descending Sequence Path}, $B(m, n)[m, n]$ also forms a shifted path. There are only two ways to merge two paths into a graph where every node has indegree and outdegree $1$: either the paths independently form cycles, or they join into one larger cycle. In other words, either both $B(m, n)[1, m-1]$ and $B(m, n)[m, n]$ are cyclic, or $B(m, n)$ is cyclic, and since $B(m, n)[1, m-1]$ clearly isn't cyclic, $B(m, n)$ is cyclic, as desired.

To deduce the shape, we can use Schensted's Theorem. Since $B(m, n)$ is unimodal, i.e.\ it is increasing up to some point and decreasing after that point, and $B(m, n)_m = n$, the LIS (and hence the first row) has length $m$ and the LDS (and hence the first column) has length $n-m+1$. Since $B(m, n)$ is a sequence of length $n$, it must be that $\Gamma\lp\lp B(m, n)\rp\rp = (m, 1, 1, \ldots, 1)$, as desired.
\end{proof}

For example, when $(m, n) = (5, 7)$, we have $B(m, n) = (2, 3, 4, 6, 7, 5, 1)$. This permutation corresponds to the SYT
\[\young(13457,2,6),\]
which has the desired shape.

We will also require use of the following construction.

\begin{definition}
For integers $1 < m \leq \frac{n}{2}$, we will define $B'(m, n)$ to be the following concatenation of sequences:
\[B'(m, n) = \lp \lc\frac{n+m}{2}\rc\rp\oplus I\lp L\lp m-1\rp\rp\oplus(1)\oplus I^m\lp L\lp n-m\rp\rp\oplus\lp\lf\frac{m}{2}\rf+1\rp.\]
\end{definition}

\begin{lemma}
\label{B'(m, n) is a permutation}
For any two integers $1 < m \le \frac{n}{2}$, the sequence $B'(m, n)$ is a permutation of length $n$.
\end{lemma}

\begin{proof}
Once again, consider $B'(m, n)$ in terms of sets of elements it contains. We can write it as containing the elements in $(1)$, $I\lp L(m-1)\rp$, $\lp\lf\frac{m}{2}\rf+1\rp$, $I^m\lp L(n-m)\rp$, and $\lp\lc\frac{n+m}{2}\rc\rp$. Writing these as sets, we obtain the following set of elements for the union of the first three:

\[ \{1\}\cup \left\{2, 3, \ldots, \lf\frac{m}{2}\rf, \lf \frac{m}{2}\rf+2, \ldots, m-1, m\right\}\cup \left\{\lf\frac{m}{2}\rf+1\right\}
    = \{1, 2, \ldots, m\}. \]

Following a similar process for our last two sequences, we obtain the set
\[ \left\{m+1, m+2, \ldots, \lc \frac{n+m}{2}-1\rc, \lc \frac{n+m}{2}+1\rc, \ldots, n-1, n\right\} \cup \left\{\lc\frac{n+m}{2}\rc\right\}
    = \{m+1, m+2, \ldots, n\}. \]
Thus, overall, we have the set $\{1, 2, \ldots, n\}$, as desired.
\end{proof}

\begin{lemma}
\label{Base Case 2 Validity}
The permutation $B'(m, n)$ is cyclic, and $\Gamma\lp P\lp B'(n)\rp\rp$ consists of $m$ $2$'s followed by $n-2m$ $1$'s.
\end{lemma}

\begin{proof}
First we prove cyclicity. We can prove this in the same way as the proof of Lemma \ref{Base Case Validity}. Since, $B(m, n)[2, m-1] = I\lp L(m-1)\rp$ and $B(m, n)[m+1, n-1] = I^m\lp L(n-m)\rp$, by Theorem \ref{Descending Sequence Path}, these two subarrays form shifted paths. Observe that we can extend these shifted paths to $B(m, n)[1, m]$ and $B(m, n)[m+1, n]$ respectively. Since neither $B(m, n)[1, m]$ nor $B(m, n)[m+1, n]$ is cyclic, $B(m, n)$ is cyclic, as desired. In terms of shape, consider the LIS and LDS of $B(m, n)$. Since $B(m, n)$ is composed of two descending sequences, it has a LIS of $2$. The LDS can be computed by splitting $B'(m, n)$ into blocks. Namely, $\lc\frac{n+m}{2}\rc$, $I\lp L\lp m-1\rp\rp\oplus(1)$, $I^m\lp L\lp n-m\rp\rp$, and $\lp\lf\frac{m}{2}\rf\rp$. Observe that our second and third blocks both form descending sequences, but all elements in the second block are less than all elements in the third block. Hence, we can use at most one of the two blocks. If we choose to use the second block, we obtain a descending sequence of length $m$, while using the third yields a descending sequence of length $n-m$. Since $m \leq \frac{n}{2}$, our LDS is of length $n-m$. Hence by Schensted's theorem, our shape consists of $m$ rows of length $2$ followed by $n-2m$ rows of length $1$, which yields the claimed shape.
\end{proof}

\begin{example} When $(m, n) = (3, 7)$, $B'(m, n) = (5, 3, 1, 7, 6, 4, 2)$. This permutation corresponds to the SYT
\[\young(12,34,56,7),\]
which has the desired shape. \end{example}

Together, $B(m,n)$ and $B'(m,n)$ establish useful base cases which will help prove the main theorem. Now, consider the following construction.

\begin{definition}
Given a cyclic permutation $\sigma$ and an an integer $k > 1$, we define $A(\sigma, k)$ to be a sequence of length $n+k$ obtained by modifying $\sigma$ as follows: first, replace $n$ in $\sigma$ with $n + \lc\frac{k}{2}\rc$, and then append $I^n\lp L(k)\rp \oplus (n)$.
\end{definition}

\begin{lemma}
For a cyclic permutation $\sigma$ and an integer $k > 1$, $A(\sigma, k)$ is a permutation.
\end{lemma}

\begin{proof}
$A(\sigma, k)[1, n]$ contains the elements $\{1, 2, \ldots, n-1\} \cup \left\{n+\lc\frac{k}{2}\rc\right\}$ and $A(\sigma, k)[n+1, n+k]$ contains $\{n, n+1, \ldots, n+k\} \setminus \left\{n+\lc\frac{k}{2}\rc\right\}$. Hence, the elements in $A(\sigma, k)$ are
\[\{1, 2, \ldots, n-1\} \cup \left\{n+\lc\frac{k}{2}\rc\right\} \cup \{n, n+1, \ldots, n+k\} \setminus \left\{n+\lc\frac{k}{2}\rc\right\} = \{1, 2, \ldots, n+k\}.\]
Thus $A(\sigma, k)$ is a permutation.
\end{proof}

\begin{lemma}
\label{K-Add Construction validity}
The permutation $A(\sigma, k)$ is cyclic, and $\Gamma\lp P\lp A(\sigma, k\rp\rp = \Gamma\lp P\lp\sigma\rp\rp + (1^k)$.
\end{lemma}

\begin{proof}
In terms of cyclicity, we once again see that $A(\sigma, k)$ is composed of two shifted paths, one containing the indices $i$ for $1 \leq i \leq n$, and the other containing the indices $i$ for $n < i \leq n+k$. Thus, since neither of these paths are cycles, $A(\sigma, k)$ is cyclic, as desired. Further, Theorem \ref{Descending Sequence Concatenation} immediately implies $\Gamma\lp P\lp A(\sigma, k\rp\rp = \Gamma\lp P(\sigma)\rp+(1^k)$.
\end{proof}

\begin{example} Consider the permutation $\sigma = (4, 1, 5, 3, 2)$, which is $B'(2, 5)$. Then
$P(\sigma)$ is the SYT \[\young(12,35,4).\]
Further, $A(\sigma, 3) = (4, 1, 7, 3, 2, 8, 6, 5)$. Then $P(A(\sigma,3))$ is the SYT
\[\young(125,368,47),\]
which has the shape specified in Lemma~\ref{K-Add Construction validity}. \end{example}

We are now ready to prove our main result, Theorem \ref{Cyclic Theorem}.

\begin{proof}[Proof of Theorem \ref{Cyclic Theorem}]


Let $\gamma$ and $\gamma'$ be two shapes, where $\gamma$ is the shape we aim to construct. Define the sequence $\chi$ by $\chi_i = \gamma_i-\gamma'_i$. Then the $A$ construction can be used to construct $\gamma$ from $\gamma'$ iff $\chi$ is nonincreasing and $\chi_1 = \chi_2$. If $\gamma_1 > \gamma_2$, we can use $B$ to construct $\gamma'$: we will let $\gamma' = B(\gamma_1-\gamma_2+1, \gamma_1-\gamma_2+|\gamma|)$, which is the shape consisting of one row of size $\gamma_1-\gamma_2+1$, followed by $|\gamma|-1$ rows of size $1$. For this $\gamma'$ and $i > 1$, it is clear that $\chi_{i+1} \leq \chi_i$. Furthermore, $\chi_1 = \gamma_1-(\gamma_1-\gamma_2+1) = \gamma_2-1$, while $\chi_1 = \gamma_2-1$ as well. Hence $\chi_1 = \chi_2$, and we can construct $\gamma$ from $\gamma'$ using $A$.

Now suppose that $\gamma_1 = \gamma_2$. In this case, $A$ can already construct $\gamma$ if we allowed it to operate on the empty shape. However, it can't, and we thus require our third shape, $B'$. We want $B'$ to generate a shape $\gamma'$, such that the sequence $\chi$ defined by $\chi_i = \gamma_i-\gamma'_i$ is both nonincreasing and satisfies $\chi_1 = \chi_2$. Since it is impossible to construct large cyclic permutations that have a shape contained in a single column, we use $B'$ to generate two column shapes. Specifically, let $q$ be the largest position such that $\gamma_1 = \gamma_q$. Clearly, $q > 1$. Since $\gamma_{q+1} < 
\gamma_q$, if we let $\gamma'$ consist of $q$ $2$'s and $|\gamma|-q$ $1$'s, $\chi$ is both nonincreasing and has equal first two rows. This means that $\gamma' = B'(q, |\gamma|+q)$.

Thus, we can construct a permutation $\sigma$ such that $\Gamma\lp P\lp\sigma\rp\rp = \gamma$ for any $\gamma$ with at least two rows and two columns, completing the proof.
\end{proof}

\section{Almost Cyclic Permutations}
\label{sec:almost}

\begin{definition}
A permutation is defined to be {\em almost cyclic} if its cycle decomposition is $(n-1, 1)$.
\end{definition}

In this section, we will prove the following theorem, which implies that the cycle decomposition $(n-1,1)$ is RSK-complete for odd $n$ but not for even $n>2$.

\begin{theorem}
\label{Almost Cyclic Theorem}
When $n$ is odd, all RS shapes apart from $(1, 1, \ldots, 1)$ and $(n)$ can be achieved using only almost cyclic permutations. When $n$ is even, all RS shapes apart from $(1, 1, \ldots, 1)$, $(n)$, and $(\frac{n}{2}, \frac{n}{2})$ can be achieved using only almost cyclic permutations.
\end{theorem}

\begin{remark}
When $n$ is even, the special case where $\gamma = (\frac{n}{2}, \frac{n}{2})$ is, in fact, unachievable for any permutation with a cycle decomposition containing a $1$, which we will show in Lemma \ref{Difference in Row Lengths Fixed Points}.
\end{remark}

The proof of Theorem \ref{Almost Cyclic Theorem} will follow a similar structure to the proof of Theorem \ref{Cyclic Theorem}. We will introduce several new base cases and continue to use $A$. First, we will need the following variation on $L(n)$.

\begin{definition}
\label{lambda' def}
Given any integer $n > 1$, let $p$ be $0$ when $n$ is even and $1$ when $n$ is odd, and let $w$ denote $\lf\frac{n}{2}\rf+2p$. We define $L'(n)$ to be the sequence
\[L'(n) = (n, n-1, \ldots, w+1, w-1, \ldots, 2, 1).\]
\end{definition}

Observe that this sequence is missing the element $w$, hence is of size $n-1$.

\begin{lemma}
\label{Almost Cyclic Path P1}
Let $t$ and $n$ be any two positive integers. Let $\sigma$ denote the sequence $(t)\oplus L'(n)$. Consider the directed graph $G = (V, E)$, where $V = \{1, 2, \ldots, n\}$ and $E = \{(i, \sigma_i):1<i\le n\}$. $G$ consists of a self loop (a cycle of size $1$) and a shifted path. 
\end{lemma}

\begin{figure}
\centering
\begin{tikzpicture}
\filldraw[fill=gray!15!white, draw=gray] (6,6) circle (0.3cm) node {4};
\draw[-latex] (5.7, 6) arc (120:420:0.6) ;
\filldraw[fill=gray!15!white, draw=gray] (4,4) circle (0.3cm) node {3};
\filldraw[fill=gray!15!white, draw=gray] (8,4) circle (0.3cm) node {5};
\draw[-latex] (7.7, 4) -- (4.3, 4);
\draw[-latex] (4.26, 3.87) -- (7.74,2.13);
\filldraw[fill=gray!15!white, draw=gray] (4,2) circle (0.3cm) node {2};
\filldraw[fill=gray!15!white, draw=gray] (8,2) circle (0.3cm) node {6};
\draw[-latex] (7.7, 2) -- (4.3, 2);
\draw[-latex] (4.26, 1.87) -- (7.74,0.13);
\filldraw[fill=gray!15!white, draw=gray] (4,0) circle (0.3cm) node {1};
\filldraw[fill=gray!15!white, draw=gray] (8,0) circle (0.3cm) node {7};
\draw[-latex] (7.7, 0) -- (4.3, 0);
\end{tikzpicture}
\caption{The directed graph corresponding to the sequence $L'(7)$.}
\label{lambda' 7}
\end{figure}
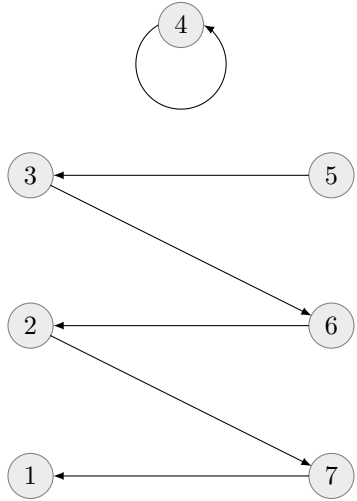

\begin{proof}
First observe that the integer $t$ does not affect the graph, and hence will not play any role in the proof (it will be useful when we invoke this result). We will separate the proof of Lemma \ref{Almost Cyclic Path P1} into two cases.
\begin{description}
    \item[Case 1] $n$ is even. In this case, $w = \frac{n}{2}$, as defined in Definition \ref{lambda' def}. Let $N(i) = \sigma_i$. Then, $N(w+1) = n+2-(w+1) = w+1$. Hence, $w+1$ forms our self loop. With this in mind, let us organize the remaining nodes in $G$ into a bipartite graph, with nodes $1$ through $w$ on the left and nodes $w+2$ through $n$ on the right. For $1 < j \leq w$, $N(j) = n+2-j$, and for $w+2 \leq j \leq n$, $N(j) = n+1-j$. Let $\ell_i = w+1-i$ and $r_i = w+1+i$.
    Then,
    \[N(\ell_i) = n+2-\ell_i = n+2-(w+1-i) = w+1+i = r_i,\]
    and
    \[N(r_i) = n+1-r_i = n+1-(w+1+i) = w+i = \ell_{i+1}.\]
    Therefore, $N(N(\ell_i)) = \ell_{i+1}$, and $N(N(r_i)) = r_{i+1}$. This means that our bipartite graph isn't cyclic and thus forms a path.
    \item[Case 2] $n$ is odd. In this case, $w = \frac{n+3}{2}$, as defined in Definition \ref{lambda' def}. Then, $N(w-1) = n+1-(w-1) = w-1$. Hence, $w-1$ forms our self loop. With this in mind, let us organize the remaining nodes in $G$ into a bipartite graph, with nodes $1$ through $w-2$ on the left and nodes $w$ through $n$ on the right. Let $\ell_i = w-1-i$ and $r_i = w-1+i$. Then, for $1 < j \leq w-2$, $N(j) = n+2-j$, and for $w \leq j \leq n$, $N(j) = n+1-j$.
    Then
    \[N(\ell_i) = n+2-\ell_i = n+2-(w-1-i) = w+i = r_{i+1},\]
    and
    \[N(r_i) = n+1-r_i = n+1-(w-1+i) = w-1-i = \ell_i.\]
    Therefore, $N(N(\ell_i)) = \ell_{i+1}$, and $N(N(r_i)) = r_{i+1}$. This means that our bipartite graph isn't cyclic and thus forms a path.
\end{description}
In both cases, we have a self loop and a path, as desired.
\end{proof}

For example, $(5) \oplus L'(7) = (5, 7, 6, 4, 3, 2, 1)$, which corresponds to the directed graph shown in Figure \ref{lambda' 7}. 

\begin{definition}
\label{D(m, n) definition}
Given positive integers $(m,n) \neq (2, 4)$ such that $1 < m \leq \frac{n}{2}$, let $p$ be $1$ if $m$ is odd and $0$ otherwise, and let $w$ denote $\lf\frac{m}{2}\rf+2p$. Then, we define $D(m,n)$ to be the following concatenation of sequences:
\[D(m, n) = \lp\lc\frac{n+m}{2}\rc\rp\oplus L'(m)\oplus I^m\lp L(n-m)\rp\oplus(w).\]
\end{definition}

\begin{lemma}
For $1 < m \leq \frac{n}{2}$, $D(m, n)$ is a permutation.
\end{lemma}

\begin{proof}
Take $w$ and $p$ to have values as defined in Definition \ref{D(m, n) definition}. Then, we will once again consider sets of elements. We have the sets $\{\lc \frac{n+m}{2}\rc\}$, $\{1, 2, \ldots, w-1, w+1, \ldots, m-1, m\}$, $\{m+1, m+2, \ldots, \lc\frac{n+m}{2}\rc-1, m+\lc\frac{n+m}{2}\rc+1, \ldots, n-1, n\}$, and $\{w\}$. The union of the second and the fourth sets is $\{1, 2, \ldots, m\}$, and the union of the first and third sets is $\{m+1, m+2, \ldots, n\}$. Thus, taken together, the union of all of our sets is $\{1, 2, \ldots, n\}$.
\end{proof}

\begin{lemma}
\label{Almost Cyclic Base Case Validity}
The permutation $D(m, n)$ is almost cyclic, and $\Gamma\lp P\lp D(m, n)\rp\rp$ consists of $m$ $2$'s followed by $n-2m$ $1$'s.
\end{lemma}

\begin{proof}
First consider cyclicity. Since $D(m, n)[1, m] = \lp\lc\frac{n+m}{2}\rc\rp\oplus L'(m)$, it forms a path and a self loop by Lemma \ref{Almost Cyclic Path P1}. Similarly, $D(m, n)[m+1, n-1] = I^m\lp L(n-m)\rp$, and forms a shifted path by Theorem \ref{Descending Sequence Path}. Observe that $D(m, n)[m+1, n]$ also forms a shifted path, but not a cycle. Since each node in our graph has in and outdegree $1$, and our paths don't form cycles on their own, $D(m, n)$ contains a cycle and a self loop, giving the desired cycle decomposition of $(n-1, 1)$.

We can, once again, use the LIS and LDS of $D(m, n)$ to determine its shape. Since $D(m, n)$ is composed of two descending sequences, it has a LIS of $2$. To determine its LDS, we split it into four blocks. Namely, $\lp\lc\frac{n+m}{2}\rc\rp$, $L'(m)$, $I^m\lp L(n-m)\rp$, and $(w)$. Once again, the second and third blocks form descending sequences, but all elements of the second block are less than all elements of the third block. Hence, only one of these two blocks may be used in any descending sequence. If we use the third block, we get a descending sequence of length $n-m$.

Now, consider using the second block. If $m > 2$, then $w > 1$ and hence the second block ends with a $1$. In this case, using the second block yields an LDS of length $m$. However, if $m = 2$, then we get an LDS of length $m+1$. Since $n > 4$ when $m = 2$, the third block still provides a longer LDS of $n-m$.

Hence, our LDS is of length $n-m$. Then, by Schensted's theorem, we have two columns and $n-m$ rows. The only valid shape consists of $m$ rows of length $2$ and $n-2m$ rows of length $1$, as claimed.
\end{proof}

\begin{example} If $(m, n) = (3, 7)$, $D(m, n) = (5, 2, 1, 7, 6, 4, 3)$. This sequence corresponds to the following tableau:
\[\young(13,24,56,7)\]
which has the shape specified in Lemma~\ref{Almost Cyclic Base Case Validity}. \end{example}

Let us also introduce a special case of $D$.

\begin{definition}
Given a positive integer $n > 1$, we define $D'(n)$ to be the following concatenation of sequences:
\[D'(n) = (w)\oplus L'(n)\]
where $w$ is defined as in Definition \ref{lambda' def}.
\end{definition}

\begin{lemma}
\label{D' permutation}
For all $n > 1$, $D'(n)$ is a permutation.
\end{lemma}

\begin{proof}
The set of elements contained in $L'(n)$ is $\{1, 2, \ldots, n\} \setminus \{w\}$, hence the set of elements contained in $D'(n)$ is $\{1, 2, \ldots, n\}$. Since $L'(n)$ is of length $n-1$, the sequence $D'(n)$ forms a permutation.
\end{proof}

\begin{lemma}
\label{Almost Cyclic Base Case Variation}
The permutation $D'(n)$ is almost cyclic and $\Gamma\lp P\lp D'(n)\rp\rp = (2, 1, \ldots, 1)$.
\end{lemma}

\begin{proof}
We can see that $D'(n)$ is almost cyclic by Lemma \ref{Almost Cyclic Path P1}, and the shape follows from Theorem \ref{Descending Sequence Concatenation}.
\end{proof}

We can now prove Theorem \ref{Almost Cyclic Theorem}.

\begin{proof}[Proof of Theorem \ref{Almost Cyclic Theorem}]
Let us again consider any shape $\gamma$ with at least two rows, at least two columns, and not of the form $(\frac{n}{2}, \frac{n}{2})$. We will separate the proof into three cases.
\begin{description}
    \item[Case 1] $\gamma = (2, 1, \ldots, 1)$. The almost cyclic permutation with this shape is $D'(n)$.
    \item[Case 2] $\gamma_1 > \gamma_2$ and $\gamma_1 > 2$. In this case, we can consider the sequence $\gamma'$ where $\gamma'_1 = \gamma_1-1$, and $\gamma'_i = \gamma_i$ for $i > 1$. This sequence clearly represents a shape with at least two columns and at least two rows, and thus by Theorem \ref{Cyclic Theorem}, we can create a cyclic permutation $\sigma$ with the shape $\gamma'$. Then the permutation $\sigma\oplus(|\sigma|+1)$ has the shape $\gamma'+(1)$, which is $\gamma$.
    \item[Case 3] $\gamma_1 = \gamma_2$. Once again, since all of our constructions create or maintain almost cyclic permutations, we will only consider shapes. $A$ adds $1$ to the length of the first $k > 1$ rows of any nonempty shape, while $D$ constructs some base case shape $\gamma'$. Consider the shape $\chi$ defined by $\chi_i = \gamma_i-\gamma_i$. We again need this to be nonincreasing with $\chi_1 = \chi_2$, and we will use the exact same base case shape $\gamma'$ as in Theorem \ref{Cyclic Theorem}. Namely, let $q$ be the largest position with $\gamma_1 = \gamma_q$. Then the shape consisting of $q$ $2$'s followed by $|\gamma|-q$ $1$'s is a valid base case, and this can be constructed by setting $\gamma' = D(q, |\gamma|+q)$.
\end{description}
In every case, we have constructed a permutation with shape $\gamma$.
\end{proof}

\section{The Converse}
\label{sec:converse}

In this section, we will complete the proof of our main theorem:

\begin{theorem}
\label{Converse}
For even $n$, the only RSK-complete cycle decomposition is $(n)$. For odd $n$, only the cycle decompositions $(n)$ and $(n-1, 1)$ are RSK-complete.
\end{theorem}

We have already proven that the cycle decomposition $(n)$ is RSK-complete for all $n$ and the cycle decomposition  $(n-1, 1)$ is RSK-complete for odd $n$ in sections~\ref{sec:cyclic} and ~\ref{sec:almost}. In this section, we will show that no other cycle decomposition is RSK-complete. First, consider the following operation on a sequence.

\begin{definition}
Given a positive integer $f$ and a sequence $\sigma$ that does not contain $f$, define $R(\sigma, f)$ to be the sequence $\sigma'$ of the same size as $\sigma$ constructed as follows:
\begin{equation}
    \sigma'_i =
        \begin{cases}
            \sigma_i & \text{if } \sigma_i < f \\
            \sigma_i-1 & \text{if } \sigma_i > f.
        \end{cases}
\end{equation}
\end{definition}

Now, let us describe the cycle decomposition of permutations that can achieve the shape $(n-1, 1)$.

\begin{lemma}
\label{Cycle Decomposition Shape Classification}
All permutations that can achieve the shape $(n-1, 1)$ have cycle decompositions of the form $(k, 1, 1, \ldots, 1)$, for some $1 < k \leq n$.
\end{lemma}

\begin{proof}
Consider any permutation $\sigma$ with the desired shape. Then, by Schensted's Theorem, the LIS of $\sigma$ is $n-1$. Thus, all but one element are in the same relative order as the identity permutation. Hence, we form $\sigma$ by performing a cyclic shift on some subrange $[i, j]$ of the identity permutation. Here, we either move element $i$ to position $j$ and move elements $i+1,\ldots, j$ to one position lower, or we move element $j$ to position $i$ and move elements $i,\ldots, j-1$ to one position higher. Then, the subrange $[i, j]$ forms a cycle while the rest of the elements are fixed points, so the resulting cycle decomposition would be $(|j-i+1|, 1, 1, \ldots, 1)$.
\end{proof}

Lemma~\ref{Cycle Decomposition Shape Classification} significantly restricts the structure of any RSK-complete cycle decompositions. We will now show that this restricted class does not contain any RSK-complete cycle decompositions other than the ones that we have already identified.

\begin{lemma}
\label{Fixed Point Breaking}
Given a permutation $\sigma \in S_n$ such that $|\Gamma\lp P\lp\sigma\rp\rp| = 2$ (i.e.\ the RSK shape of $\sigma$ has two rows), and an integer $1 \le i \le n$ such that $\sigma_i = i$ (i.e.\ $i$ is a fixed point of the permutation), we must have  $\sigma_j < i$ for all $1 \leq j < i$.
\end{lemma}

\begin{proof}
Let $\sigma_j$ be the largest element for $1 \leq j < i$. By the pigeonhole principle, if $\sigma_j > i$, there is some other $j'$ such that $i < j' \leq n$ and $\sigma_{j'} < i$. This would then create a descending sequence $(j, i, j')$ of length $3$. Then, by Schensted's Theorem, $|\Gamma\lp P\lp\sigma\rp\rp| \geq 3$.
\end{proof}

\begin{lemma}
\label{Difference in Row Lengths Fixed Points}
Any permutation $\sigma$ with $|\Gamma\lp P\lp\sigma\rp\rp| = 2$ and whose cycle decomposition contains $r$ $1$'s must have $\Gamma\lp P\lp\sigma\rp\rp_1 \geq \Gamma\lp P\lp\sigma\rp\rp_2+r$.
\end{lemma}

\begin{proof}
We will prove this inductively. Clearly, this is true for $r = 0$. For $r > 0$, consider any $i$ for which $\sigma_i = i$. Then, by Lemma \ref{Fixed Point Breaking}, for all $1 \leq j \leq n$, if $j < i$, $\sigma_j < \sigma_i$, and if $j > i$, $\sigma_j > i$. If we let $\sigma' = R(\sigma[1, i-1]\oplus\sigma[i+1, n], i)$, this means that the LIS of $\sigma$ is one greater than the LIS of $\sigma'$. Thus, by Schensted's theorem, $\Gamma\lp P\lp\sigma\rp\rp_1 = \Gamma\lp P\lp\sigma'\rp\rp_1+1$, completing the induction.
\end{proof}

Finally, we can prove Theorem \ref{Converse}.

\begin{proof}[Proof of Theorem \ref{Converse}]
We will separate the proof of Theorem \ref{Converse} into two cases.
\begin{description}
    \item[Case 1] $n$ is even. In this case, by Lemma \ref{Cycle Decomposition Shape Classification}, all permutations, apart from cyclic ones, that have a shape of $(n-1, 1)$ have a cycle decomposition containing at least $1$ self loop. Thus, by Lemma \ref{Difference in Row Lengths Fixed Points}, such a cycle decomposition cannot have a permutation with shape $(\frac{n}{2}, \frac{n}{2})$, as desired.
    \item[Case 2] $n$ is odd. In this case, by Lemma \ref{Cycle Decomposition Shape Classification}, all permutations, apart from cyclic and almost cyclic ones, that have a shape of $(n-1, 1)$ have a cycle decomposition containing at least $2$ self loops. Thus, by Lemma \ref{Difference in Row Lengths Fixed Points}, such a cycle decomposition cannot have a permutation with shape $(\frac{n+1}{2}, \frac{n-1}{2})$.
\end{description}

Thus, for even $n$, only cyclic permutations can achieve all shapes, and for odd $n$, only cyclic and almost cyclic permutations can achieve all shapes, completing the proof.
\end{proof}

\bibliographystyle{alpha}
\bibliography{refs}
\end{document}